\pdfoutput=1
\documentclass{amsart}
\usepackage{amsmath,
 amssymb,
 fullpage,
 mathtools,
 mathpazo,
 thmtools,
 enumerate, moreenum
}
\usepackage[
 pagebackref,
 colorlinks,
 allcolors=blue
]{hyperref}

\usepackage[capitalize]{cleveref}

\newcommand{\reg}{\mathrm{reg}}
\newcommand{\sub}{\mathrm{sub}}
\newcommand{\St}{\mathrm{St}}
\newcommand{\Nov}{\mathrm{Nov}}

\newlength{\bibitemsep}
\newlength{\bibparskip}
\let\oldthebibliography\thebibliography
\renewcommand\thebibliography[1]{%
 \oldthebibliography{#1}%
 \setlength{\parskip}{\bibparskip}%
 \setlength{\itemsep}{\bibitemsep}%
}
\newenvironment{smatrix}{\left(\begin{smallmatrix}}{\end{smallmatrix}\right)}

\newcommand{\CC}{\mathbf{C}}

\newcommand{\cB}{\mathcal{B}}

\newcommand{\cS}{\mathcal{S}}
\newcommand{\cW}{\mathcal{W}}

\renewcommand{\le}{\leqslant}

\renewcommand{\ge}{\geqslant}

\newcommand{\into}{\hookrightarrow}

\DeclareMathOperator{\Hom}{Hom}

\DeclareMathOperator{\GL}{GL}
\DeclareMathOperator{\SO}{SO}
\DeclareMathOperator{\GSp}{GSp}
\DeclareMathOperator{\GSpin}{GSpin}

\newcommand{\dfour}[4]{\begin{smatrix} #1 \\ &#2 \\ &&#3 \\ &&&#4 \end{smatrix}}
\newcommand{\stbt}[4]{\begin{smatrix}#1&#2\\#3&#4\end{smatrix}}
\usepackage{bbold}
\newcommand{\triv}{\mathbb{1}}

\newtheorem{definition}{Definition}[section]
\newtheorem{theorem}[definition]{Theorem}
\newtheorem{proposition}[definition]{Proposition}
\newtheorem{lemma}[definition]{Lemma}
\newtheorem{conjecture}{Conjecture}

\newtheorem{corollary}[theorem]{Corollary}
\newtheorem{lettertheorem}{Theorem}

\theoremstyle{remark}
\declaretheorem[name=Remark,sibling=theorem,qed={\lower-0.3ex\hbox{$\diamond$}}]{remark}

\begin{document}

\title{On local zeta-integrals for $\GSp(4)$ and $\GSp(4) \times \GL(2)$}
\author{David Loeffler}
\begin{abstract}
 We prove that Novodvorsky's definition of local $L$-factors for generic representations of $\GSp(4) \times \GL(2)$ is compatible with the local Langlands correspondence when the $\GL(2)$ representation is non-supercuspidal. We also give an interpretation in terms of Langlands parameters of the ``exceptional'' poles of the $\GSp(4) \times \GL(2)$ $L$-factor, and of the ``subregular'' poles of $\GSp(4)$ $L$-factors studied in recent work of R\"osner and Weissauer; and deduce consequences for Gan--Gross--Prasad type branching laws, either for reducible generic representations, or for irreducible but non-generic representations.
\end{abstract}

\maketitle

\section{Introduction}

 In this note, we study the local $L$-factors associated to irreducible smooth representations $\pi \times \sigma$ of the group $\GSp(4, F) \times \GL(2, F)$, where $F$ is a nonarchimedean local field of characteric 0 (corrsponding to the natural 8-dimensional representation of the $L$-group). These $L$-factors can be defined in several possible ways. Firstly, one can use the local Langlands correspondence of \cite{gantakeda11}; secondly, one can use Shahidi's method. Thirdly, supposing $\pi$ and $\sigma$ to be generic, one can use a local zeta-integral of Rankin--Selberg type introduced by Novodvorsky \cite{novodvorsky79}. It is shown in \cite{gantakeda11} that the first two constructions agree, and we shall denote the resulting $L$-factor simply by $L(\pi \times \sigma, s)$. However, it is not obvious whether the $L$-factor $L^{\Nov}(\pi \times \sigma, s)$ defined via Novodvorsky's integral agrees with $L(\pi \times \sigma, s)$.

 \begin{conjecture}
  \label{conj:compat}
  For any generic irreducible representations $\pi$ of $\GSp_4(F)$ and $\sigma$ of $\GL_2(F)$, we have $L(\pi \times \sigma, s) = L^{\Nov}(\pi \times \sigma, s)$.
 \end{conjecture}

 The Novodvorsky integral formula plays a key role in our recent work with Pilloni et al \cite{LPSZ1} on the $p$-adic interpolation of $L$-values for cuspidal automorphic representations of $\GSp_4$ and $\GSp_4 \times \GL_2$, which gives a further incentive to study \cref{conj:compat}. The conjecture is known to hold in a substantial range of cases by work of Soudry \cite{Soudry-GSpGL}, which we recall as \cref{thm:soudry} below, but many other cases still remain open.

 \subsection{Compatibility of $L$-factors}

  Our first new result is the following:

  \begin{lettertheorem}\label{thm:compat}
   \cref{conj:compat} holds under the additional assumption that the $\GL(2, F)$-representation $\sigma$ be non-super\-cuspidal.
  \end{lettertheorem}

  The case of $\sigma$ an irreducible principal series was established in \cite[Theorem 8.9(i)]{LPSZ1}, so it remains to consider the case when $\sigma$ is a special representation. Twisting $\pi$ appropriately, we can assume that $\sigma = \St$ is the Steinberg representation, and the proof in this case will be given as \cref{thm:steinberg} below.

  Since this paper was initially posted on the Mathematics ArXiv, a complementary result was proved by Yao Cheng \cite{cheng21-SO5}, showing that \cref{conj:compat} also holds if $\sigma$ is supercuspidal and $\pi$ has trivial central character (so $\pi$ factors through $\operatorname{PGSp}_4(F) \cong \SO_5(F)$). In particular, combining Cheng's result and \cref{thm:compat} of the present paper proves \cref{conj:compat}, for any $\sigma$, if the central character of $\pi$ is a square in the group of characters of $F^\times$; this is Theorem 1.3 of \cite{cheng21-SO5}. We are optimistic that combining the methods of this paper and \cite{cheng21-SO5} may lead to a complete proof of \cref{conj:compat} in the near future.

 \subsection{Exceptional poles for $\GSp(4) \times \GL(2)$}

  In the analysis of Novodvorsky's $L$-factor, an important role is played by a partition of the set of its poles into \emph{regular} and \emph{exceptional} poles (\cref{def:Novexpole}). Let $\pi$ and $\sigma$ be as in \cref{conj:compat}. One sees easily that a necessary condition for $s_0 \in \CC$ to be an exceptional pole of $L(\pi \times \sigma, s)$ is that $\chi_{\pi} \chi_{\sigma} |\cdot|^{2s_0} = 1$. We propose the following conjecture:

  \begin{conjecture}
   \label{conj:Novexpole}
   If $s_0 \in \CC$ is such that $\chi_{\pi} \chi_{\sigma} |\cdot|^{2s_0} = 1$, then $s_0$ is an exceptional pole of $L^{\Nov}(\pi \times \sigma, s)$ if and only if it is a pole of the ratio
   \[ \frac{L(\pi \times \sigma, s) L(\pi \times \sigma, s + 1)}{L(\pi \times \sigma \times \St, s + \tfrac{1}{2})}.\]
   Equivalently (by \cref{lem:langlandsfactor} below), $s_0$ is an exceptional pole if and only if the 8-dimensional Weil--Deligne representation $\phi_{\pi} \otimes \phi_{\sigma}$ has a 1-dimensional unramified direct summand whose $L$-factor has a pole at $s_0$.
  \end{conjecture}

  Our second new result, whose proof is intertwined with that of \cref{thm:compat}, is the following:

  \begin{lettertheorem}\label{thm:Novexpole}
   Conjecture $\beta$ holds under the additional ssumption that $\sigma$ be non-supercuspidal.
  \end{lettertheorem}

 \subsection{Subregular poles for $\GSp(4)$}

  In order to prove Theorems A and B, we shall use a relation between Novodvorsky's zeta-integral for $\GSp(4) \times \GL(2)$ and a zeta-integral for $\GSp(4)$ studied by Piatetski-Shapiro \cite{piatetskishapiro97}, depending on a choice of (split) Bessel model of $\pi$. R\"osner and Weissauer \cite{roesnerweissauer17, roesnerweissauer18} have computed the Piatetski-Shapiro $L$-factors for all generic $\pi$, and verified that they coincide with the Langlands $L$-factors (independently of the choice of Bessel model). In their computations, an important role is played by the notion of a \emph{subregular} pole of the $\GSp(4)$ $L$-factor (see \cref{def:subreg} below). The proof of our main theorems also gives a conceptual interpretation of subregular poles, which may be of independent interest:

  \begin{lettertheorem}\label{thm:subregpole}
   Let $\pi$ be a generic irreducible representation of $\GSp(4, F)$ with central character $\chi_{\pi}$; and let $s_0 \in \CC$. Then $s_0$ is a subregular pole of $L(\pi, s)$ (for some choice of split Bessel model) if and only if one of the following two possibilities occurs:
   \begin{enumerate}
   \item $s_0$ is a pole of the ratio $\dfrac{L(\pi, s) L(\pi, s + 1)}{L(\pi \times \St, s + \tfrac{1}{2})}$; equivalently, the Langlands parameter of $\pi$ has a 1-dimensional unramified direct summand whose $L$-factor has a pole at $s_0$. In this case, we necessarily have $\chi_{\pi} |\cdot|^{2s_0 + 1} \ne 1$. \medskip
   \item $\chi_{\pi} |\cdot|^{2s_0 + 1} = 1$ and $s_0 +\tfrac{1}{2}$ is an exceptional pole of $L(\pi \times \St, s)$; equivalently, the Langlands parameter of $\pi$ has a 2-dimensional, self-dual direct summand isomorphic to an unramified twist of the Steinberg parameter, whose $L$-factor has a pole at $s_0$.
   \end{enumerate}
  \end{lettertheorem}

  That is, a pole is subregular precisely when it arises from a direct summand of the Langlands parameter which is either 1-dimensional, or 2-dimensional and self-dual.

  \begin{remark}
   \cref{thm:subregpole} is a fairly straightforward consequence of the results of \cite{roesnerweissauer18}. We include it here partly because it motivates the formulation of Conjectures \ref{conj:Novexpole} and \ref{conj:dist}, and more importantly, because \cref{thm:subregpole} plays a major role in the proof of \cref{thm:compat}. More precisely, we shall prove directly that an analogue of \cref{thm:subregpole} holds with the Langlands $L$-factor in the denominator replaced by the Novodvorsky $L$-factor, and deduce \cref{thm:compat} when $\sigma$ is the Steinberg by comparing this with \cref{thm:subregpole}.
  \end{remark}

 \subsection{Distinction of representations}

  Our next result is an interpretation of exceptional poles in terms of $H$-invariant periods, where $H = \{(h_1, h_2) \in \GL(2, F) \times \GL(2, F) : \det(h_1) = \det(h_2)\}$ (which is naturally a subgroup of $\GSp(4, F)$, see \cref{sect:notation} below). It is not hard to show (see \cref{prop:auxzeta} below) that if $s_0$ is an exceptional pole of $L(\pi \times \sigma, s)$, then we have $\Hom_H\Big( \pi \otimes (|\cdot|^{s_0} \boxtimes \sigma), \CC\Big) \ne 0$.

\refstepcounter{conjecture}

  \begin{conjecture}
   \label{conj:dist}
   The dimension of $\Hom_H\Big( \pi \otimes (|\cdot|^{s_0} \boxtimes \sigma), \CC\Big)$ is 1 if $s_0$ is an exceptional pole of $L^{\Nov}(\pi \times \sigma, s)$, and 0 otherwise.
  \end{conjecture}

  \begin{lettertheorem}\label{thm:dist}
   \cref{conj:dist} is true if at least one of the following conditions holds:
   \begin{itemize}
    \item $\sigma$ is non-supercuspidal,
    \item the central character of $\pi$ is a square.
   \end{itemize}
  \end{lettertheorem}

  \begin{remark}
   The combination of Conjectures \ref{conj:Novexpole} and \ref{conj:dist} is closely related to the Gan--Gross--Prasad conjecture for non-tempered representations formulated in \cite{gangrossprasad20}.

   More precisely, taking $s_0 = 0$, Conjectures \ref{conj:Novexpole} and \ref{conj:dist} predict that $\Hom_H\big(\pi \otimes (\triv \boxtimes \sigma), \CC\big)$ is non-zero if and only if the $\GSp_4$-valued Weil--Deligne representation $\phi_{\pi}$ contains $\phi_{\sigma}^\vee$ as a self-dual direct summand. If we suppose $\chi_\pi = \chi_\sigma = 1$, so the representations involved factor through $\SO_5$ and $\SO_4$, then this condition on the Weil--Deligne representations is equivalent to the Langlands parameters of $\pi$ and $\triv \boxtimes \sigma^\vee$ forming a ``relevant pair'' in the sense of \cite{gangrossprasad20}. According to the conjectures of \emph{op.cit.}, this should be a necessary and sufficient condition for $\Hom_H\big(\pi \otimes (\triv \boxtimes \sigma), \CC\big)$ to be non-zero.\footnote{In \emph{op.cit.}~it is also assumed that the $L$-parameters are ``of Arthur type'', which in this situation corresponds to assuming that $\pi$ and $\sigma$ are tempered; but this is not essential to the formulation of the conjecture. It suffices that $\pi$ and $\sigma$ are generic (or members of generic $L$-packets).}

   So, in the light of \cref{thm:dist}, \cref{conj:Novexpole} is an instance of the non-tempered Gan--Gross--Prasad conjectures (mildly generalised from orthogonal groups to spin groups); and \cref{thm:Novexpole} verifies the conjecture for representations of this type when $\sigma$ is non-supercuspidal.
  \end{remark}

 \subsection{Multiplicity one for reducible representations}

  We now give an interpretation of the above results in terms of branching laws for reducible representations. It follows from results of Prasad and Emory--Takeda\footnote{The restriction $(\sigma_1 \boxtimes \sigma_2) |_{H}$ is a direct sum of irreducible $H$-representations lying in the same $L$-packet. Theorem 5 of \cite{prasad96} shows that there is at most one representation $\tau$ in this $L$-packet such that $\Hom_H(\pi \otimes \tau, \CC) \ne 0$; and the general result on multiplicity-one for GSpin groups from \cite{emorytakeda21}, via the isomorphisms $\GSp_4 \cong \GSpin_5$ and $H \cong \GSpin_4$, shows that for any such $\tau$ the Hom-space has dimension $\le 1$, giving the claim. Alternatively, the multiplicity-one result can be extracted directly from the  proof of \cite[Theorem 5]{prasad96} (Prasad, pers.comm.), although the result is not explicitly stated there.} that we have $\dim \Hom_H(\pi \otimes (\sigma_1 \boxtimes \sigma_2), \CC) \le 1$ for any irreducible generic representations $\pi$ of $\GSp(4, F)$ and $\sigma_1$, $\sigma_2$ of $\GL(2, F)$. Of course, this Hom-space can only be non-zero if $\chi_{\pi} \chi_{\sigma_1} \chi_{\sigma_2} = 1$.

  We consider here the situation in which one or both of the $\sigma_i$ is replaced by the reducible principal-series representation $\Sigma$ having the Steinberg representation as subrepresentation. (However, we continue to assume that $\pi$ itself is irreducible and generic.) One checks easily that for any irreducible generic $\sigma$ with $\chi_{\pi} \chi_{\sigma} = 1$, the leading term at $s = 0$ of the zeta-integral defining $L^{\Nov}(\pi \times \sigma, s)$ gives a non-zero element of $\Hom_H(\pi \otimes( \Sigma \boxtimes \sigma), \CC)$. Similarly, if $\chi_{\pi} = 1$, then the leading term of Piatetski-Shapiro's zeta integral (with $\lambda_1 = \lambda_2 = 1$ in the notation of \cref{sect:bessel}) defines a nonzero element of $\Hom_H(\pi \otimes( \Sigma \boxtimes \Sigma), \CC)$. We conjecture that these Hom-spaces are actually 1-dimensional, giving a generalisation to $\GSp_4 \times \GL_2 \times \GL_2$ of the results on branching laws for reducible representations proved in \cite{harrisscholl01} and \cite{loeffler-ggp}:

  \begin{conjecture}\label{conj:branching} \
   \begin{enumerate}[label=(\alph*)]
    \item Suppose $\pi$ and $\sigma$ are irreducible and generic, with $\chi_{\pi} \chi_{\sigma} = 1$. Then $\Hom_H(\pi \otimes (\Sigma \boxtimes \sigma), \CC)$ is 1-dimensional (and hence the leading term of the Novodvorsky zeta-integral is a basis of this space).
    \item Suppose $\pi$ is irreducible and generic with $\chi_{\pi} = 1$. Then the space $\Hom_H(\pi \otimes (\Sigma \boxtimes \Sigma), \CC)$ is 1-dimensional (and hence the leading term of the Piatetski-Shapiro zeta-integral is a basis).
   \end{enumerate}
  \end{conjecture}

  We shall see in \S 9 below that \cref{conj:branching}(a) implies \cref{conj:dist}, and we shall prove the following partial result:

  \begin{lettertheorem}\label{thm:branching} \
  \begin{enumerate}[label=(\alph*)]
   \item \cref{conj:branching}(a) is true if at least one of the following two conditions holds:
   \begin{enumerate}[label=(\roman*)]
    \item $\chi_{\pi}$ is a square in the group of characters of $F^\times$;
    \item $\sigma$ is non-supercuspidal, and $s = 0$ is not an exceptional pole of $L^\Nov(\pi \times \sigma, s)$.
   \end{enumerate}\medskip
   \item \cref{conj:branching}(b) is true.
  \end{enumerate}
  \end{lettertheorem}

  These results are used in \cite{LZ20} and \cite{LZ21-BSD} to study Euler systems for Shimura varieties attached to $\GSp(4)$ and $\GSp(4) \times \GL(2)$.

 \subsection{Acknowledgements} It is a pleasure to thank Mirko R\"osner and Rainer Weissauer for several interesting exchanges relating to the theory of \cite{roesnerweissauer17,roesnerweissauer18}; and Kei Yuan Chan, Yao Cheng, and Dipendra Prasad for their remarks on an earlier version of this paper. Finally, the author would also like to thank the anonymous referee for their careful reading of the paper and numerous valuable comments and corrections.

\section{General notation}
 \label{sect:notation}

 We shall consider the following setting:

 \begin{itemize}
  \item $F$ is a nonarchimedean local field of characteristic 0, and $q$ is the cardinality of its residue field.
  \item $|\cdot|$ the absolute value on $F$, normalised by $|\varpi| = \tfrac{1}{q}$ for $\varpi$ a uniformizer.

  \item We fix a nontrivial additive character $e: F \to \CC^\times$.

  \item $G$ denotes the group $\GSp(4, F)$ of matrices preserving the standard anti-diagonal symplectic form, and $H$ the group $\{ (h_1, h_2) \in \GL(2, F) \times \GL(2, F): \det(h_1) = \det(h_2)\}$. We consider $H$ as a subgroup of $G$ via the embedding
  \[ \iota: \left(\begin{pmatrix} a & b \\ c & d\end{pmatrix}, \begin{pmatrix} a' & b'\\ c'& d'\end{pmatrix}\right)\mapsto \begin{smatrix} a &&& b\\ & a' & b' & \\ & c' & d' & \\ c &&& d \end{smatrix}.
  \]

  \item In this paper ``representation'' will mean an admissible smooth representation on a complex vector space.

  \item An ``$L$-factor'' will mean a function of $s \in \CC$ of the form $1 / P(q^{-s})$, where $P$ is a polynomial with $P(0) = 1$. Any fractional ideal of $\CC[q^s, q^{-s}]$ containing the unit ideal is generated by a unique $L$-factor.
 \end{itemize}

\section{Principal series representations of
 \texorpdfstring{$\GL(2)$}{GL(2)}}

 \subsection{Definitions}

  \begin{definition}
   For $\mu$, $\nu$ smooth characters $F^\times \to \CC^\times$, and $s \in \CC$, we write $i_s(\mu, \nu)$ for the space of smooth functions $f: \GL(2, F) \to \CC$ satisfying
   \[ f\left(\begin{smatrix} a & \star \\ 0 & d \end{smatrix} g\right) = \mu(a) \nu(d) |a/d|^s f(g), \]
   with $\GL(2, F)$ acting via right translation. If $s = \tfrac{1}{2}$ we write simply $i(\mu, \nu)$.
  \end{definition}

  As is well known, $i(\mu, \nu)$ is irreducible unless $\mu / \nu = |\cdot|^{\pm 1}$; if $\mu / \nu = |\cdot|$ it has a 1-dimensional quotient, and if $\mu / \nu = |\cdot|^{-1}$ it has a 1-dimensional subrepresentation. There is a unique (up to scalars) non-zero intertwining operator $i_s(\mu, \nu) \to i_{1-s}(\nu, \mu)$. The \emph{Steinberg representation} $\operatorname{St}$ is the unique irreducible subrepresentation of $i(|\cdot|^{1/2}, |\cdot|^{-1/2})$.

 \subsection{Godement--Siegel sections}

  Let $\cS(F^2)$ denote the Schwartz space of locally-constant, compactly-supported functions on $F^2$, with $\GL(2, F)$ acting via the usual formula $(g \cdot \Phi)(x, y) = \Phi( (x, y) \cdot g)$. Then we define
  \[
   f^{\Phi}(g; \mu, \nu, s) = \mu(\det g) |\det g|^s \int_{F^\times} \Phi((0, x) \cdot g) (\mu/\nu)(x) |x|^{2s}\, \mathrm{d}^\times x,
  \]
  which converges for $\Re(s) > 0$ and defines an element of $i_s(\mu, \nu)$. We write simply $f^{\Phi}(\mu, \nu, s)$ for the function $f^{\Phi}(-; \mu, \nu, s)$. We may extend the definition to all $s \in \CC$ by analytic continuation, away from simple poles at the $s$ such that $|\cdot|^{2s} = \nu / \mu$. 

  \begin{remark}
   We have $f^{\Phi}(g; \mu, \nu, s) = \mu(\det g) f^{\Phi}(g; \mu/\nu, s)$ in the notation of \cite[\S 8.1]{LPSZ1}.
  \end{remark}

  \begin{proposition} Let $\widehat{\Phi}$ denote the Fourier transform.
   \begin{enumerate}[label=(\roman*)]
    \item If $\nu \ne 1$, then the map $\Phi \mapsto f^{\Phi}(1, \nu, 0)$ is well-defined, nonzero, and $\GL(2, F)$-equivariant, and identifies $i(|\cdot|^{-1/2}, |\cdot|^{1/2} \nu)$ with the maximal quotient of $\cS(F^2)$ on which $F^\times$ acts by $\nu$.
    \item If $\nu \ne |\cdot|^{-2}$, then the map $\Phi \mapsto f^{\widehat\Phi}(\nu, 1, 1)$ is is well-defined, nonzero, and $\GL(2, F)$-equivariant, and identifies $i(|\cdot|^{1/2}\nu, |\cdot|^{-1/2})$ with the maximal quotient of $\cS(F^2)$ on which $F^\times$ acts by $\nu$.
   \end{enumerate}
  \end{proposition}

   \begin{proof} Well-known.
%
   \end{proof}

 \subsection{Whittaker functions}
  \label{sect:gl2whittaker}

  For $\Phi \in \cS(F^2)$ and $\mu, \nu$ smooth characters, we define
  \[ W^{\Phi}(g; \mu, \nu, s) = \int_{F} f^{\Phi}(\begin{smatrix} 0 & 1 \\ -1 & 0 \end{smatrix}\begin{smatrix} 1 & x \\ 0 & 1 \end{smatrix}g; \mu, \nu, s) e(x) \, \mathrm{d}x, \]
  and $W^{\Phi}(g; \mu, \nu) = W^{\Phi}(g; \mu, \nu, \tfrac{1}{2})$. Again we write simply $W^{\Phi}(\mu, \nu)$ for the function $W^{\Phi}(-; \mu, \nu)$. Note that the integral is entire as a function of $s$, although $f^{\Phi}(-)$ may not be, and there is no $s$ such that $W^{\Phi}(g; \mu, \nu, s)$ vanishes for all $g$ and $\Phi$. We have
  \[ W^{\Phi}(\mu, \nu, s) = \varepsilon \cdot W^{\widehat\Phi}(\nu, \mu,1- s)\]
  where $\varepsilon$ is a nonzero constant independent of $\Phi$ (a local root number). We want to study the space of functions $W^{\Phi}(\mu, \nu)$ for varying $\Phi \in \cS(F^2)$.
  \begin{itemize}

  \item If $\sigma = i(\mu, \nu)$ is irreducible, then the space of functions $W^{\Phi}(\mu, \nu)$ for varying $\Phi \in \cS(F^2)$ is precisely the Whittaker model\footnote{We define all Whittaker models for $\GL(2, F)$ with respect to the character $\begin{smatrix}1 & x \\ &1 \end{smatrix} \mapsto e(-x)$; this is slightly non-standard, but will simplify our formulae later} $\cW(\sigma)$ of $\sigma$.

  \item If $\sigma$ has a one-dimensional quotient, then the functions $f^{\Phi}(\mu, \nu, s)$ are regular at $s = \tfrac{1}{2}$ and span the representation $\sigma$; and mapping $f^{\Phi}$ to $W^{\Phi}$ gives a bijection from $\sigma$ to a subspace $\cW(\sigma) \subset \operatorname{Ind}_{N_2}^{\GL_2} e^{-1}$, containing the Whittaker model of the generic subrepresentation $\sigma^{\mathrm{gen}}$ as a codimension-1 subspace.

  \item If $\sigma$ has a one-dimensional subrepresentation, then it does not have a Whittaker model; and the functions $W^{\Phi}(\mu, \nu)$ instead give the Whittaker model of $\sigma' = i(\nu, \mu)$, as we have just defined it. In this case, the $f^{\Phi}(\mu, \nu, s)$ are not all well-defined at $s = \tfrac{1}{2}$ (they may have poles). If we define
  \[ \cS_0(F^2) \coloneqq \{ \Phi \in \cS(F^2) : \Phi(0, 0) = 0\},\]
  then the $f^{\Phi}$ for $\Phi \in \cS_0(F^2)$ are well-defined and span $\sigma$. The corresponding $W^{\Phi}$ span the Whittaker model of the irreducible subrepresentation of $\sigma'$, which is also the irreducible quotient of $\sigma$.
  \end{itemize}

\section{Bessel models}

 Throughout this section, $\pi$ denotes an irreducible representation of $G$ with central character $\chi_{\pi}$.

 \subsection{The Bessel model}
  \label{sect:bessel}
  Let $\Lambda = (\lambda_1, \lambda_2)$ be a pair of characters of $F^\times$ with $\lambda_1 \lambda_2 = \chi_{\pi}$. A (split) \emph{Bessel model} of $\pi$ (with respect to $\Lambda$) is a $G$-invariant subspace isomorphic to $\pi$ inside the space of functions $G \to \CC$ satisfying
  \[ B( \begin{smatrix} 1 & & u & v\\ & 1 & w & u \\ & & 1 \\ &&&1 \end{smatrix}\dfour{x}{y}{x}{y} g) = e(u) \lambda_1(x) \lambda_2(y) B(g). \]
  It follows from \cite[Theorem 6.3.2(i)]{robertsschmidt16} that if such a subspace exists, it is unique, and we denote it by $\cB_{\Lambda}(\pi)$.

 \subsection{Piatetski-Shapiro's integral}

  Suppose $\pi$ admits a $\Lambda$-Bessel model $\cB_{\Lambda}(\pi)$.

  \begin{definition}
   For $B \in\cB_{\Lambda}(\pi)$, $\mu$ a smooth character of $F^\times$, and $\Phi_1, \Phi_2 \in \cS(F^2)$, we define
   \[ Z(B, \Phi_1, \Phi_2; \Lambda, \mu, s) = \int_{N_H \backslash H} B(h) \Phi_1((0, 1) \cdot h_1) \Phi_2( (0, 1) \cdot h_2) \mu(\det h)|\det h|^{s + 1/2} \, \mathrm{d}h, \]
   where $N_H = \left(\stbt{1}{\star}{0}{1}, \stbt{1}{\star}{0}{1}\right)$ is the unipotent radical of the standard Borel subgroup of $H$.
  \end{definition}

  This converges for $\Re(s) \gg 0$ and has meromorphic continuation as a rational function of $q^s$. If $\mu$ is the trivial character, we write simply $Z(B, \Phi_1, \Phi_2; \Lambda, s)$; we can always reduce to this case by replacing $\pi$ with $\pi \otimes \mu$, and $(\lambda_1, \lambda_2)$ with $(\lambda_1 \mu, \lambda_2 \mu)$. The following is the main result of  \cite{roesnerweissauer17}:

  \begin{theorem}[R\"osner--Weissauer]
   \label{thm:RW1}
   The $\CC$-vector space spanned by the functions
   \[
    \{ Z(B, \Phi_1, \Phi_2; \Lambda, s) : B \in \cB_{\Lambda}(\pi), \Phi_1, \Phi_2 \in \cS(F^2)\}
   \]
   is a fractional ideal of $\CC[q^s, q^{-s}]$ containing the constant functions. This ideal is independent of $\Lambda$, and is generated by the $L$-factor $L(\pi, s)$ associated to the Langlands parameter $\phi_{\pi}$.
  \end{theorem}

 \subsection{Generic representations}

  Recall that $\pi$ is said to be \emph{generic} if it admits a Whittaker model, i.e.~if it is isomorphic to a $G$-invariant subspace of the space of functions $W: G \to \CC$ satisfying
  \begin{equation}
   \label{eq:whittaker}
   W(\begin{smatrix} 1 &\phantom{-}x & \phantom{-}* & \phantom{-}* \\ & \phantom{-}1 & \phantom{-}y & \phantom{-}* \\ &&\phantom{-}1 & -x \\ &&&\phantom{-}1 \end{smatrix} g) = e(x + y) W(g).
  \end{equation}
  Such a model is unique if it exists; we denote it by $\cW(\pi)$.

  \begin{proposition}
   Suppose $\pi$ is generic, and let $\mu$ be a smooth character of $F^\times$. For any $W \in \cW(\pi)$, the integral
   \[ B(W; \mu, s) \coloneqq \int_{F^\times} \int_F W\left(\begin{smatrix} a \\ &a \\ &x&1 \\ &&&1 \end{smatrix}w_2\right)|a|^{s-3/2} \mu(a)\, \mathrm{d}x\, \mathrm{d}^\times a,\qquad w_2 = \begin{smatrix}1 \\ &&1 \\ &-1  \\ &&&1\end{smatrix} \]
   converges for $\Re(s) \gg 0$ and has meromorphic continuation as a rational function of $q^{s}$. The set $\{ B(W; \mu, s) : W \in \cW(\pi)\}$ is a fractional ideal of $\CC[q^s, q^{-s}]$ containing the constant functions, and it is generated by the spinor $L$-factor $L(\pi \times \mu, s)$ associated to the Langlands parameter of $\pi \times \mu$.
  \end{proposition}

  \begin{proof} The definition of the integral, and the proof of its analytic continuation, are due to Novodvorsky \cite{novodvorsky79}. The proof that the $L$-factor defined by this integral coincides with the Langlands $L$-factor is due to Takloo-Bighash \cite{takloobighash00}.\end{proof}

  \begin{proposition}[Roberts--Schmidt]
   For any $s$, the space of functions
   \[ \widetilde{B}_{W}(g; \mu, s) \coloneqq \frac{1}{L(\pi \times \mu, s)} B(g W; \mu, s)\]
   for $W \in \cW(\pi)$ is the Bessel model $\cB_{\Lambda}(\pi)$, where $\Lambda = (\mu^{-1}|\cdot|^{1/2 - s}, \mu \chi_{\pi}|\cdot|^{s-1/2})$.
  \end{proposition}

  See \cite{robertsschmidt16} for details. Since $\mu$ is arbitrary, we  see that a generic representation has a Bessel model for every character $\Lambda$ with $\lambda_1\lambda_2 = \chi_{\pi}$.

 \subsection{Exceptional and subregular poles}

  Suppose $\pi$ admits a $\Lambda$-Bessel model.

  \begin{definition}
   We define $L^{\Lambda}_{\reg}(\pi, s)$ and $L^{\Lambda}_{\mathrm{Kir}}(\pi, s)$ as the unique $L$-factors such that
   \begin{align*}
    \left( \left\{Z(B, \Phi_1, \Phi_2; \Lambda, s) : B \in \cB_{\Lambda}(\pi), \Phi_1, \Phi_2 \in \cS(F^2), \Phi_1(0, 0) \Phi_2(0, 0) = 0\right\}\right) &= \left(L^{\Lambda}_{\reg}(\pi, s)\right)\\
    \left(\left\{ Z(B, \Phi_1, \Phi_2; \Lambda, s) : B \in \cB_{\Lambda}(\pi), \Phi_1, \Phi_2 \in \cS(F^2), \Phi_1(0, 0) =\Phi_2(0, 0) = 0\right\}\right)&= \left(L^{\Lambda}_{\mathrm{Kir}}(\pi, s)\right)
   \end{align*}
   We let $L^{\Lambda}_{\mathrm{ex}}(\pi, s) = L(\pi, s) / L^{\Lambda}_{\reg}(\pi, s)$, and $L^{\Lambda}_{\sub}(\pi, s) = L^{\Lambda}_{\reg}(\pi, s) / L^{\Lambda}_{\mathrm{Kir}}(\pi, s)$, which are clearly also $L$-factors, so we have
   \[ L(\pi, s) = L^{\Lambda}_{\mathrm{ex}}(\pi, s) \cdot L^{\Lambda}_{\sub}(\pi, s) \cdot L^{\Lambda}_{\mathrm{Kir}}(\pi, s).\]
   The poles of $L^{\Lambda}_{\mathrm{ex}}(\pi, s)$ are said to be \emph{exceptional poles} for $\pi$ and $\Lambda$; the poles of $L^{\Lambda}_{\sub}$ are said to be \emph{subregular poles}.
  \end{definition}

  \begin{remark}
   The factor we call $L^{\Lambda}_{\mathrm{Kir}}(\pi, s)$ is denoted by $L(s, M)$ in the works of R\"osner--Weissauer, where $M$ is a certain auxiliary space. The notation $L^{\Lambda}_{\mathrm{Kir}}(\pi, s)$ is intended to emphasise the relation with Kirillov models.
  \end{remark}

  \begin{theorem}[Piatetski-Shapiro, {\cite[Theorem 4.3]{piatetskishapiro97}}]
   If $\pi$ is generic, then $L^{\Lambda}_{\mathrm{ex}}(\pi, s)$ is identically 1, for all possible choices of $\Lambda$.
  \end{theorem}

  So exceptional poles do not occur for generic representations; however, we shall see later that subregular poles do frequently occur. The poles of $L^{\Lambda}_{\sub}(\pi, s)$ (if any) are simple \cite[Corollary 3.2]{roesnerweissauer18}. We say $s = s_0$ is a \emph{type I subregular pole} if it is a pole of the ratio
  \[
   \frac{Z(B, \Phi_1, \Phi_2; \Lambda, s)}{L^{\Lambda}_{\mathrm{Kir}}(\pi, s)}
  \]
  for some $(\Phi_1, \Phi_2)$ with $\Phi_1(0,0) = 0$, and a \emph{type II subregular pole} if we may take $(\Phi_1, \Phi_2)$ such that $\Phi_2(0,0) = 0$. Clearly, any subregular pole must be of type I or type II (but these possibilities are not mutually exclusive).

  Since the two factors of $H$ are conjugate in $G$, one checks that $s_0$ is a type II subregular pole for the $(\lambda_1, \lambda_2)$ Bessel model if and only if it is a type I subregular pole for the $(\lambda_2, \lambda_1)$ Bessel model. So it suffices to analyse type II subregular poles. Moreover, if $s_0$ is a type II subregular pole, then it must also be a pole of $L(\lambda_1, s + \tfrac{1}{2})$ (cf.~Proposition 3.1 of \cite{roesnerweissauer18}; note that the characters $\rho$ and $\rho^{\divideontimes}$ of \emph{op.cit.} are $\lambda_2$ and $\lambda_1$ in our notation -- the order is switched since we use a different matrix model of $\GSp_4$). In particular, for a given $\pi$ whose $L$-factor has a pole at $s_0$, there is at most one character $\Lambda$ such that $s_0$ is a type II subregular pole for the $\Lambda$-Bessel model, namely $\Lambda = \left(|\cdot|^{-1/2-s_0}, \chi_\pi |\cdot|^{1/2+s_0}\right)$.

  \begin{definition}
   \label{def:subreg}
   Suppose $\pi$ is generic. We shall simply say ``$s_0$ is a subregular pole of $L(\pi, s)$'' to mean that it is a type II subregular pole for this specific Bessel character, or (equivalently) a type I subregular pole for the character given by swapping $\lambda_1$ and $\lambda_2$.
  \end{definition}

  Note that these two Bessel characters coincide if and only if $\chi_{\pi} |\cdot|^{2s_0 + 1} = 1$.

  The subregular poles have been tabulated for all Bessel models in \cite{roesnerweissauer17, roesnerweissauer18}. Non-supercuspidal representations of $\GSp(4, F)$ have been classified by Sally and Tadi\'c \cite{sallytadic93}, into 11 types I--XI; the tables in \cite[Appendix A]{robertsschmidt07} are a useful reference. All types except I, VII, and X have several subtypes, with subtypes ``a'' being the generic representations. So the generic non-supercuspidal representations are those of Sally--Tadic types $\{$I, IIa, IIIa, IVa, Va, VIa, VII, VIIIa, IXa, X, XIa$\}$. We can neglect the supercuspidal representations and those of types $\{$VII, VIIIa, IXa$\}$, since $L(\pi, s)$ is identically 1 for all such representations.

  \begin{theorem}[R\"osner--Weissauer]
   \label{thm:RW2}
   If $\pi$ is a generic representation, then every pole of $L(\pi, s)$ is subregular, unless $\pi$ is of type IIIa or IVa, in which case there are no subregular poles.\qed
  \end{theorem}

\section{Zeta integrals for \texorpdfstring{$\GSp(4) \times \GL(2)$}{GSp(4) x GL(2)}}

 \subsection{Novodvorsky's integral}

  We now suppose $\pi$ is a generic irreducible representation of $G$; and we let $\sigma$ be a representation of $\GL_2(F)$ which is either irreducible and generic, or a reducible principal-series representation with one-dimensional quotient, defining the Whittaker model $\cW(\sigma)$ in the latter case as in \cref{sect:gl2whittaker} above

  For $W_0 \in \cW(\pi)$, $\Phi_1 \in \cS(F^2)$, and $W_2 \in \cW(\sigma)$, we define
  \[
   Z(W_0, \Phi_1, W_2; s) = \int_{Z_G N_H \backslash H} W_0(\iota(h))f^{\Phi_1}\left(h_1; 1, (\chi_\pi \chi_\sigma)^{-1}, s\right) W_2(h_2)\ \mathrm{d}h.
  \]

  \begin{theorem}[Novodvorsky]
   There is $R < \infty$, depending on $\pi$ and $\sigma$, such that the integral converges for $\Re(s) > R$ and has analytic continuation as a rational function in $q^s$. The $\CC$-vector space spanned by the functions $Z(W_0, \Phi, W_2; s)$ for varying $(W_0, \Phi, W_2)$ is a fractional ideal of $\CC[q^s, q^{-s}]$ containing the constant functions.\qed
  \end{theorem}

  See \cite{novodvorsky79}, \cite{Soudry-GSpGL}, and \cite[\S 8]{LPSZ1} for further details.

  \begin{definition}
   We let $L^{\Nov}(\pi \times \sigma, s)$ be the unique $L$-factor generating the fractional ideal of values of the zeta integral.
  \end{definition}

  This is the $L$-factor featuring in \cref{conj:compat}. Although the conjecture is open in general, many cases can be obtained from the following result of Soudry. If $\tau_1, \tau_2$ are irreducible generic representations of $\GL(2, F)$ with the same central character, then we can regard the product $\tau_1 \boxtimes \tau_2$ as a representation of the group
  \[
   (\GL(2, F) \times \GL(2, F)) / \{ (z, z^{-1}): z \in F^\times\}.
  \]
  This group is isomorphic to the split orthogonal similitude group $\operatorname{GSO}(4, F)$, and there is a theta-lifting from this group to $\GSp(4, F)$. The non-supercuspidal generic representations that are $\theta$-lifts from $\operatorname{GSO}(2, 2)$ are those of Sally--Tadi\'c types I, IIa, Va, VIa, VIIIa, X and XIa, while types IIIa, IVa, VII and IXa are not in the image. The image of the $\theta$-lift also contains some (but not all) of the generic supercuspidal representations of $\GSp(4)$.

  \begin{theorem}[{Soudry, \cite{Soudry-GSpGL}}]
   \label{thm:soudry}
   Suppose that $\pi$ is an irreducible generic representation of the form $\pi = \theta(\tau_1, \tau_2)$, where $\tau_i$ are irreducible generic representations of $\GL(2, F)$ as above. Suppose that $\sigma$ is irreducible, and if $\sigma$ is supercuspidal, that it is not an unramified twist of $\tau_1^\vee$ or $\tau_2^\vee$. Then we have
   \[ L^{\Nov}(\pi \times \sigma, s) = L(\pi \times \sigma, s) = L(\tau_1 \times \sigma, s) L(\tau_2 \times \sigma, s), \]
   where $L(\tau_i \times \sigma, s)$ are the $\GL_2 \times \GL_2$ Rankin--Selberg $L$-factors.\qed
  \end{theorem}

 \subsection{An auxiliary integral} To better understand Novodvorsky's integral, we write it in terms of the following auxiliary function:

  \begin{definition}
   For $W_0 \in \cW(\pi)$ and $W_2 \in \cW(\sigma_2)$ we define
   \[
    Z(W_0, W_2, s) \coloneqq \int_{N_2 \backslash \GL_2} W_0\left(\begin{smatrix} \det g \\ &g \\ &&1 \end{smatrix}\right) W_2(g) |\det g|^{s-1}\, \mathrm{d} g,
   \]
   where $N_2$ is the upper-triangular unipotent subgroup of $\GL(2, F)$.
  \end{definition}

  One computes that the function on $h$ defined by $h \mapsto Z(h W_0, h_2 W_2, s)$ depends only on the first projection $h_1$ of $h$, and belongs to the principal-series $\GL(2, F)$-representation $i_{1-s}(1, \nu^{-1})$, where $\nu = (\chi_\pi \chi_{\sigma})^{-1}$.

  \begin{proposition}
   \label{prop:auxint1}
   For $W_0, W_2$ as above and $\Phi \in \cS(F^2)$, we have
   \[ Z(W_0, \Phi_1, W_2; s) = \left\langle Z(W_0, W_2; s), f^{\Phi}(1, \nu, s) \right\rangle, \]
   where $\langle- ,- \rangle$ denotes the canonical duality pairing between $i_{1-s}(1, \nu^{-1})$ and $i_s(1, \nu)$, given by integration over $B_2 \backslash \GL_2$.
  \end{proposition}

  \begin{proof}
   Let $H_+$ be the subgroup $\{(h_1, h_2) \in H : h_1 \text{ is upper-triangular}\}$ of $H$. Then $Z_G N_H \le H_+$, and we can write the integral over $Z_G N_H \backslash H$ defining $Z(W_0, \Phi, W_2; s)$ as an integral over $Z_G N_H \backslash H_+$ composed with an integral over $H_+ \backslash H$. However, the map $\GL(2, F) \to H_+$ given by $\gamma \mapsto \left(\stbt {\det \gamma}{}{}{1}, \gamma\right)$ gives a bijection $Z_G N_H \backslash H_+ \cong N_2 \backslash \GL_2$; and projection onto the first factor clearly identifies $H_+ \backslash H$ with $B_2 \backslash \GL_2$.
  \end{proof}

 \subsection{Exceptional poles of the $\GSp(4) \times \GL(2)$ integral}

  \begin{definition}
   \label{def:Novexpole}
   We define $L^{\Nov}_{\reg}(\pi \times \sigma, s)$ to be the $L$-factor generating the fractional ideal
   \[
    \big\{Z(W_0, \Phi_1, W_2; s) : W_0 \in \cW(\pi), \Phi_1 \in \cS_0(F^2), W_2 \in \cW(\sigma) \big\},
   \]
   and we define $L^{\Nov}_{\mathrm{ex}}(\pi \times \sigma, s)$ to be the quotient, so that
   \[ L^{\Nov}(\pi \times \sigma, s) = L^{\Nov}_{\reg}(\pi \times \sigma, s)L^{\Nov}_{\mathrm{ex}}(\pi \times \sigma, s).\]
  \end{definition}
  (We use implicitly here the fact that the fractional ideal $(\star)$ contains the constant functions, which follows from the proof of \cite[Theorem 8.9(i)]{LPSZ1}.)

  \begin{proposition}
   The $L$-factor $L^{\Nov}_{\reg}(\pi \times \sigma, s)$ is also the $L$-factor generating the fractional ideal
   \[ \big\{Z(W_0, W_2; s) : W_0 \in \cW(\pi), W_2 \in \cW(\sigma) \big\}. \]
  \end{proposition}

  \begin{proof}
   This follows from the formula of \cref{prop:auxint1}, since the functions $f^{\Phi}(1, \nu, s)$ for $\Phi \in \cS_0(F^2)$ are entire and span the whole of $i_{s}(1, \nu)$.
  \end{proof}

  \begin{corollary}
   \label{prop:auxzeta}
   The poles of $L^{\Nov}_{\mathrm{ex}}(\pi \times \sigma, s)$, if any, are simple. If $s = s_0$ is a pole of this factor, then we must have $\chi_{\pi} \chi_\sigma|\cdot|^{2s_0} = 1$, and
   \[ \Hom_H\left( \pi \otimes (|\cdot|^{s_0} \boxtimes \sigma), \CC\right) \ne 0. \]
  \end{corollary}

  \begin{proof}
   It follows from the previous proposition that if $Z(W_0, \Phi, W_2; s) / L^{\Nov}_{\reg}(\pi \times \sigma, s)$ has a pole of order $n \ge 1$ at $s = s_0$, for some pen $(W_0, \Phi, W_2)$, then $f^{\Phi}(1, \nu, s)$ must also have a pole of order $n$ at $s_0$ (where $\nu = (\chi_\pi \chi_\sigma)^{-1}$ as above). This can only occur if $n = 1$ and $|\cdot|^{2s_0} = \nu$. Moreover, since the residues of the $f^{\Phi}$ land in the one-dimensional representation $|\cdot|^{s_0}$, the residue at an exceptional pole defines a non-zero element of $\Hom_H\left( \pi \otimes (|\cdot|^{s_0} \boxtimes \sigma), \CC\right)$.
  \end{proof}

 \subsection{Regular poles}

  We now relate $L^{\Nov}_{\reg}(\pi \times \sigma, s)$ to the supercuspidal support of $\pi$ and $\sigma$. Recall that an irreducible $G$-representation $\pi$ is said to have \emph{supercuspidal support in $P$}, for a parabolic $P \subseteq G$, if it is a subquotient of the parabolic induction of a supercuspidal representation of the Levi of $P$. There are four conjugacy classes of parabolic subgroups in $G = \GSp(4, F)$: the whole group, the \emph{Klingen} and \emph{Siegel} parabolics
  \[ P_{\mathrm{Kl}} = \begin{smatrix} \star & \star & \star & \star \\   & \star & \star & \star \\
    & \star & \star & \star \\ & & & \star\end{smatrix}\qquad\text{and}\qquad P_{\mathrm{Si}} = \begin{smatrix} \star & \star & \star & \star \\  \star & \star & \star & \star \\
        &  & \star & \star \\ & & \star & \star\end{smatrix}\]
  and the standard Borel $B_G = P_{\mathrm{Sieg}} \cap P_{\mathrm{Kl}}$.

  \begin{proposition}
   For any $W_0$ and $W_2$, we have
   \[ Z(W_0, W_2, s) = \int_{B_2 \backslash \GL_2} Y(g W_0, g W_2, s)\, \mathrm{d}g, \]
   where
   \[ Y(W_0, W_2, s) = \int_{F^\times \times F^\times} W_0(\dfour{xy^2}{xy}{y}{1})  W_2(\stbt{x}{}{}{1})\chi_{\sigma}(y) |x|^{s - 2} |y|^{2s-2} \, \mathrm{d}^\times x\, \mathrm{d}^\times y.\]
  \end{proposition}

  \begin{proof}
   This follows by writing $B_2$ as the semidirect product of $N_2$ and the maximal torus $T_2 \cong F^\times \times F^\times$.
  \end{proof}

  Since $B_2 \backslash \GL_2$ is compact, the fractional ideal of $\CC[q^{\pm s}]$ generated by $Z(W_0, W_2, s)$ for all $(W_0, W_2)$ is contained in that generated by the functions $Y(W_0, W_2, s)$. So we need to investigate the possible asymptotic behaviour of the function $(x, y) \mapsto W_0(\dfour{xy^2}{xy}{y}{1}) W_2(\stbt{x}{}{}{1})$, for $W_0 \in \cW(\pi)$ and $W_2 \in \cW(\sigma)$. It follows from Lemma 2.6.2 of \cite{robertsschmidt07} that the support of this function is contained in a compact subset of $F \times F$, so the poles of the $Y(W_0, W_2, s)$, if any, arise from asymptotics as $x \to 0$ or $y \to 0$.

  \begin{proposition} \
   \begin{itemize}

    \item If $\pi$ is supercuspidal, or its supercuspidal support lies in the Siegel parabolic, then the support of $y \mapsto W_0(\dfour {y^2}{y}{y}{1})$ is compact in $F^\times$, for all $W_0 \in \cW(\pi)$.
    \item If $\pi$ is supercuspidal, or its supercuspidal support lies in the Klingen parabolic, then the support of $x \mapsto W_0(\dfour {x}{x}{1}{1})$ is compact in $F^\times$ for all $W_0 \in \cW(\pi)$.
    \item If $\sigma$ is supercuspidal, then the support of $x \mapsto W_2(\stbt x {}{} 1)$ is compact in $F^\times$, for any $W_2 \in \cW(\sigma)$.
   \end{itemize}
  \end{proposition}

  \begin{proof}
   We prove the first claim; the other two are similar. Let $N_{\mathrm{Kl}}$ denote the unipotent radical of $P_{\mathrm{Kl}}$. The hypotheses imply that $J_{\mathrm{Kl}}(\pi) = 0$, where $J_{\mathrm{Kl}}(\pi)$ is the Jacquet functor. As a vector space $J_{\mathrm{Kl}}(\pi) = \pi / \pi(N_{\mathrm{Kl}})$, where $\pi(N_{\mathrm{Kl}})$ is the span of vectors of the form $(n - 1) v$ for $v \in \pi$ and $n \in N_{\mathrm{Kl}}$. However, one computes easily using \eqref{eq:whittaker} that if $W_0 = (n - 1) W_0'$ for some $W_0' \in \cW(\pi)$ and $n \in N_{\mathrm{Kl}}$, then $W_0\left(\dfour{y^2}{y}{y}{1}\right) = (e(ty) - 1) W_0'\left(\dfour{y^2}{y}{y}{1}\right)$, where $t \in F$ is the $(1, 2)$-entry of $n$. If we choose $y$ small enough, then $e(ty) = 1$; so for all such $y$ we have $W_0\left(\dfour{y^2}{y}{y}{1}\right) = 0$.
  \end{proof}

  \begin{proposition}
   \label{prop:allexceptional}
   Suppose that \emph{either}
   \begin{itemize}
   \item $\pi$ is supercuspidal,
   \item $\sigma$ is supercuspidal, and $\pi$ is not a subquotient of a representation induced from the Klingen parabolic of the form $\chi \rtimes \tau$, with $\tau$ an unramified twist of $\sigma^\vee$.
   \end{itemize}
   Then $L^{\Nov}_{\reg}(\pi \times \sigma, s) = 1$, so all poles of $L^{\Nov}(\pi \times \sigma, s)$ are exceptional.
  \end{proposition}

  \begin{proof}
   If $\pi$ is supercuspidal, or $\sigma$ is supercuspidal and $\pi$ is supported in the Siegel parabolic, then the above results show that $W_0(\dfour{xy^2}{xy}{y}{1}) W_2(\stbt{x}{}{}{1})$ has compact support for all $(W_0, W_2)$, so the integrals $Y(W_0, W_2, s)$ have no poles, and hence the $Z(W_0, W_2, s)$ \emph{a fortiori} have no poles either.

   This leaves the more delicate case when $\sigma$ is supercuspidal, and $\pi$ is supported in the Klingen parabolic. The above arguments show that, if $s_0$ is a pole of $L^{\Nov}_{\reg}(\pi \times \sigma, s)$, then the leading term of $Z(W_0, W_2, s)$ at $s_0$ vanishes when $W_0 \in \cW(\pi)(N_{\mathrm{Kl}})$. Hence the leading term depends only on the image of $W_0$ in the Klingen Jacquet module of $\pi$; and this leading term defines a non-zero linear functional on $J_{\mathrm{Kl}}(\pi) \otimes \sigma$ which is $\GL(2, F)$-equivariant, up to an unramified twist, where we regard $\GL(2, F)$ as a subgroup of the Klingen Levi $F^\times \times \GL(2, F)$. Hence some unramified twist of $\sigma^\vee$ appears in the Jacquet module, and the result follows.
  \end{proof}

\section{Relating the zeta integrals}

 We'll fix throughout this section a generic irreducible representation $\pi$ of $G$.

 \subsection{The basic formula}

  The following is Proposition 8.4 of \cite{LPSZ1}:
  \begin{proposition}
   \label{prop:relation}
   For any smooth characters $\mu_2, \nu_2$ of $F$, we have
   \[
    Z(W_0, \Phi_1, W^{\Phi_2}(\mu_2, \nu_2); s) =
    L(\pi \times \nu_2, s) Z(\widetilde{B}_{W_0}, \Phi_1, \Phi_2; \Lambda, \mu_2, s),
   \]
   where $\Lambda$ is the character $\left(\chi_{\pi} \nu_2 |\cdot|^{s - \tfrac{1}{2}}, \nu_2^{-1} |\cdot|^{\tfrac{1}{2} - s}\right)$, and $\widetilde{B}_{W_0} =  \widetilde{B}_{W_0}(g; \nu_2, s) \in \cB_{\Lambda}(\pi)$.
  \end{proposition}

  Here $W^{\Phi_2}(-; \mu_2, \nu_2)$ is the Whittaker function defined in \cref{sect:gl2whittaker}.

  \begin{corollary}
   \label{cor:relation}
   If $\sigma = i(\mu_2, \nu_2)$ is a principal-series representation with $\mu_2 / \nu_2 \ne |\cdot|^{-1}$, then we have
   \[ L^{\Nov}(\pi \times \sigma, s) = L(\pi \times \mu_2, s)L(\pi \times \nu_2, s).\]
  \end{corollary}

  \begin{proof}
   Since the functions $W^{\Phi_2}(-; \mu_2, \nu_2)$ for varying $\Phi_2$ form the Whittaker model $\cW(\sigma)$, the $L$-factor $L^{\Nov}(\pi \times \sigma, s)$ is the unique $L$-factor generating the fractional ideal $ \{ Z(W_0, \Phi_1, W^{\Phi_2}(\mu_2, \nu_2); s) : W_0 \in \cW(\pi), \Phi_1, \Phi_2 \in \cS(F^2)\}$. On the other hand, the map $W_0 \mapsto \widetilde{B}_{W_0}$ is an isomorphism $\cW(\pi) \cong \cB_{\Lambda}(\pi)$, so the fractional ideal $\{ Z(\widetilde{B}_{W_0}, \Phi_1, \Phi_2; \Lambda, \mu_2, s) : W_0 \in \cW(\pi), \Phi_1, \Phi_2 \in \cS(F^2)\}$ is generated by $L(\pi \times \mu_2, s)$ by \cref{thm:RW1}.
  \end{proof}

  In particular, this shows that \cref{conj:compat} holds if $\sigma$ is an irreducible principal series (this is Theorem 8.9(i) of \cite{LPSZ1}); and we have chosen our definition of $\cW(\sigma)$, when $\sigma$ is a reducible principal series, in order to make the same statement also be valid in the reducible case.

 \subsection{Exceptional poles: the principal-series case}

  \begin{proposition}
   \label{prop:PScase}
   Suppose $\sigma = i(\mu_2, \nu_2)$ with $\mu_2 / \nu_2 \ne |\cdot|^{\pm 1}$, so $\sigma$ is an irreducible principal series.

   For $s_0 \in \CC$, we have $\chi_{\pi} \chi_\sigma |\cdot|^{2s_0} = 1$ if and only if $L(\lambda_1 \mu_2, s + \tfrac{1}{2})$ has a pole at $s = s_0$, where $(\lambda_1, \lambda_2) = \left(\chi_{\pi} \nu_2 |\cdot|^{s_0 - \tfrac{1}{2}}, \nu_2^{-1} |\cdot|^{\tfrac{1}{2} - s_0}\right)$ as above. If this condition is satisfied, then $s = s_0$ is an exceptional pole of $L^{\Nov}(\pi \times \sigma, s)$ if and only if it is a subregular pole of $L(\pi \times \mu_2, s)$.
  \end{proposition}

  \begin{proof}
   This is clear from the same argument as \cref{cor:relation}.
  \end{proof}

 \subsection{Exceptional poles: the Steinberg case}

  We now consider the formula of \cref{prop:relation} with $\mu_2 = 1$ and $\nu_2 = |\cdot|$, so that $\sigma = i(\mu_2, \nu_2)$ is reducible with 1-dimensional subrepresentation, and its unique irreducible quotient is the twist $\operatorname{St} \otimes |\cdot|^{1/2}$ of the Steinberg representation. We write $W^{\Phi_2}$ for $W^{\Phi_2}(\mu_2, \nu_2)$; hence the space of functions $W^{\Phi_2}$ for $\Phi \in \cS(F^2)$ is the Whittaker model of $\sigma' = i(\nu_2, \mu_2)$, and the $W^{\Phi_2}$ with $\Phi \in \cS_0(F^2)$ is the Whittaker model of $\operatorname{St} \otimes |\cdot|^{1/2}$.

  We are interested in the following three fractional ideals of $\CC[q^s, q^{-s}]$:
  \begin{align*}
   \label{eq:idealA}
    I &\coloneqq \left( \frac{Z(W_0, \Phi_1, W^{\Phi_2}; s)}{L(\pi, s) L(\pi, s+1)}: W_0 \in \cW(\pi), \Phi_1 \in \cS(F^2), \Phi_2 \in \cS(F^2) \right) \\
    J &\coloneqq \left( \frac{Z(W_0, \Phi_1, W^{\Phi_2}; s)}{L(\pi, s) L(\pi, s+1)} : W_0 \in \cW(\pi), \Phi_1 \in \cS(F^2), \Phi_2 \in \cS_0(F^2) \right) \\
    K &\coloneqq \left( \frac{Z(W_0, \Phi_1, W^{\Phi_2}; s)}{L(\pi, s) L(\pi, s+1)} : W_0 \in \cW(\pi), \Phi_1 \in \cS_0(F^2), \Phi_2 \in \cS_0(F^2) \right)
  \end{align*}

  \cref{cor:relation} shows that $I$ is the unit ideal. On the other hand, from the definitions of the $\GSp_4 \times \GL_2$ $L$-factors, we have
  \[
   J = \left(\frac{L^{\Nov}(\pi \times \operatorname{St}, s + \tfrac{1}{2})}{L(\pi, s) L(\pi, s+1)} \right), \quad
   K = \left(\frac{L^{\Nov}_{\reg}(\pi \times \operatorname{St}, s + \tfrac{1}{2})}{L(\pi, s) L(\pi, s+1)} \right).
  \]
  Since clearly $I \supseteq J \supseteq K$, we see that $J$ and $K$ are integral ideals (not just fractional ideals) of $\CC[q^{\pm s}]$. 

  \begin{proposition}
   The ideal $K$ vanishes at $s_0$ if and only if $s_0$ is a subregular pole of $L(\pi, s)$ (in the sense of Definition \ref{def:subreg}).
  \end{proposition}

  \begin{proof}
   This follows from \cref{prop:relation}, together with the definition of subregular poles.
  \end{proof}

  \begin{remark}
   It is \emph{not} true that the order of vanishing of $K$ at $s_0$ coincides with the order of the pole of $L_{\sub}^{\Lambda}(\pi, s)$ at $s = s_0$, where $\Lambda$ is the Bessel character $\left(|\cdot|^{-1/2-s_0}, \chi_\pi |\cdot|^{1/2+s_0}\right)$. The order of pole of $L_{\sub}^{\Lambda}(\pi, s)$ is always either 0 or 1, as we have seen; but the orders of vanishing of $J$ and $K$ can be $> 1$ in some cases. (This difference arises because $L_{\sub}$ detects the infinitesimal behaviour of Piatetski-Shapiro's integrals as $s$ varies for a fixed $\Lambda$, but the ideals $J$ and $K$ detect the behaviour along a one-parameter family in which $s$ and $\Lambda$ both vary.)
  \end{remark}

  \begin{corollary}\label{cor:AimpliesC}
   If $s_0 \in \CC$ is such that $\chi_{\pi} |\cdot|^{2s_0 + 1} \ne 1$, then $s_0$ is a subregular pole of $L(\pi, s)$ if and only if it is a pole of the ratio $\dfrac{L(\pi, s) L(\pi, s+1)}{L^{\Nov}\left(\pi \times \operatorname{St}, s + \tfrac{1}{2}\right)}$.
  \end{corollary}

  \begin{proof}
   If $\chi_{\pi} |\cdot|^{2s_0 + 1} \ne 1$, then $s_0$ cannot be a pole of $L^{\Nov}_{\mathrm{ex}}(\pi \times \operatorname{St}, s + \tfrac{1}{2})$. So the orders of vanishing of $J$ and $K$ at $s = s_0$ are the same, and the result follows from the previous proposition.
  \end{proof}

  \begin{proposition}\label{cor:AimpliesB}
   Suppose $\chi_{\pi} |\cdot|^{2s_0 + 1} = 1$. Then $J$ does not vanish identically at $s = s_0$. Hence $s=s_0$ is a subregular pole if and only if it is a pole of $L^{\Nov}_{\mathrm{ex}}(\pi \times \operatorname{St}, s + \tfrac{1}{2})$.
  \end{proposition}

  \begin{proof}
   The symmetry condition on $s_0$ shows that if $J$ vanishes identically, then the same is true if we interchange $\Phi_1$ and $\Phi_2$. Hence $\frac{Z(W_0, \Phi_1, W^{\Phi_2}; s)}{L(\pi, s) L(\pi, s+1)}$ in fact vanishes for all $\Phi_1, \Phi_2$ satisfying $\Phi_1(0, 0) \Phi_2(0, 0) = 0$. This shows that $s_0$ is an exceptional pole of the Piatetski-Shapiro $L$-factor, and such poles cannot occur for generic representations as we have seen above.
  \end{proof}

  Note that \cref{cor:AimpliesB} shows that part (1) of \cref{thm:subregpole} is true, assuming \cref{thm:compat}. Similarly, \cref{cor:AimpliesC} shows that conditions (i) and (ii) of \cref{thm:subregpole} are equivalent.

\section{Compatibility with the Langlands parameters}

 \subsection{Langlands parameters}

  Let $\rho$ be a Frobenius-semisimple Weil--Deligne representation $\operatorname{WD}(F) \to \GL(n,\CC)$. Then we can write $\rho$ (uniquely up to isomorphism) in the form
  \[ \rho = \bigoplus_i \rho_i \otimes \operatorname{sp}(n_i), \]
  where $n_i \ge 1$ are integers and $\rho_i$ are irreducible representations of the Weil group (with trivial monodromy action), such that $\sum_i n_i \dim(\rho_i) = n$. Here $\operatorname{sp}(j)$ denotes the $(j-1)$-st symmetric power of the Langlands parameter of the Steinberg representation of $\GL_2$, which is the 2-dimensional representation with Frobenius acting as $\begin{smatrix} q^{-1/2} \\ & q^{1/2}\end{smatrix}$ and monodromy as $\begin{smatrix}1 & 1\\ & 1 \end{smatrix}$. Note that we have
  \[ L(\rho, s) = \prod_i L(\rho_i, s + \tfrac{n_i-1}{2}).\]

  \begin{lemma}
   \label{lem:langlandsfactor}
   With the above notations, we have
   \[ \frac{L(\rho, s) L(\rho, s+1)}{L(\rho \times \operatorname{sp}(2), s + \tfrac{1}{2})} = \prod_{\{i : n_i = 1\}} L(\rho_i, s),\]
   and similarly
   \[ \frac{L(\rho \otimes \operatorname{sp}(2), s) L(\rho \otimes \operatorname{sp}(2), s + 1)}{L\left(\rho \otimes \operatorname{sp}(2) \otimes \operatorname{sp}(2), s + \tfrac{1}{2}\right)} = \prod_{\{i : n_i = 2\}} L(\rho_i, s).\]
  \end{lemma}

  \begin{proof}
   This is a straightforward computation using the fact that
   \[ \operatorname{sp}(n) \otimes \operatorname{sp}(2) =
    \begin{cases}
     \operatorname{sp}(n+1) \oplus \operatorname{sp}(n-1) & \text{if $n \ge 2$},\\
     \operatorname{sp}(2) & \text{if $n = 1$}.
    \end{cases}\qedhere
   \]
  \end{proof}

  We shall apply this to the 4-dimensional representations arising from the local Langlands correspondence for $G$ \cite{gantakeda11}; we write $\phi_{\pi}$ for the Langlands parameter of $\pi$, which we consider as a 4-dimensional Weil--Deligne representations by composing with the inclusion $\GSp(4, \CC) \into \GL(4, \CC)$. We also have the local Langlands correspondence $\sigma \mapsto \phi_{\sigma}$ for $\GL(2, F)$. We refer to \cite[\S 2.4]{robertsschmidt07} for an explicit description of $\phi_{\pi}$ for non-supercuspidal $\pi$.

  \begin{proposition}
   If $\pi$ is supercuspidal, or if $\sigma$ is supercuspidal and $\pi$ is not a subquotient of the Klingen parabolic induction of an unramified twist of $\sigma^\vee$, then \cref{conj:compat} implies \cref{conj:Novexpole}.
  \end{proposition}

  \begin{proof}
   I claim that under these hypotheses, the Langlands $L$-factor $L(\pi \times \sigma, s)$ has at most simple poles, and these all arise from one-dimensional summands of $\phi_{\pi} \otimes \phi_{\sigma}$.

   This claim implies the proposition, since (assuming \cref{conj:compat}), \cref{conj:Novexpole} in this case amounts to the assertion that all poles of the Novodvorsky $L$-factor are exceptional, which is true by \cref{prop:allexceptional}.

   Let us now prove the claim. First, we suppose $\sigma$ is supercuspidal. In this case, $\phi_\sigma$ is an irreducible 2-dimensional representation of the Weil group (with trivial monodromy action). If $L(\pi \times \sigma, s)$ has any poles, then $\phi_\pi$ must have one or more direct summands isomorphic to unramified twists of $\phi_\sigma^\vee \otimes \operatorname{sp}(j)$, for some $j$. However, if there is a summand with $j > 1$, or more than one such summand, then this implies that $\pi$ is a subquotient of the induction of some twist of $\sigma^\vee$ (using the explicit description of the Langlands correspondence for non-supercuspidal representations described in \S 2.4 of \cite{robertsschmidt07}), contradicting our assumptions. In the remaining case, when there is precisely one such summand and it has $j = 1$, the corresponding summand of the tensor product also has trivial monodromy, as required.

   Now let us suppose $\pi$ is supercuspidal. Then $\phi_{\pi}$ is either irreducible of dimension 4, or is the direct sum of two \emph{distinct} 2-dimensional irreducible representations (with the same determinant). So the $L$-factor is trivial unless $\sigma$ is also supercuspidal, and we may argue as before.
  \end{proof}

  \subsection{Proof of \cref{thm:compat} for Steinberg $\sigma$}

  The results of the previous section give a complete characterisation of the poles of the ratio $\frac{L(\pi,\ s) L(\pi,\ s+1)}{L^{\Nov}\left(\pi \times \St,\ s + \tfrac{1}{2}\right)}$: they are precisely the complex numbers $s_0$ such that $\chi_{\pi} |\cdot|^{2s_0 + 1} \ne 1$ and $L(\pi, s)$ has a subregular pole. We shall use this, together with the tables of subregular poles in \cite{roesnerweissauer17, roesnerweissauer18}, to compute $L^{\Nov}(\pi \times \St, s)$, and hence prove \cref{thm:compat} of the introduction.

  \begin{theorem}[\cref{thm:compat}]\label{thm:steinberg}
   Let $\pi$ be a generic irreducible representation of $\GSp(4, F)$. Then \cref{conj:compat} holds for $\sigma$ the Steinberg representation, i.e. we have
   \[ L^{\Nov}(\pi \times \St, s) = L(\pi \times \St, s).\]
  \end{theorem}

  \begin{proof}

   We can assume that $\pi$ is either supercuspidal, or that its Sally--Tadi\'c type is one of $\{$IIIa, IVa, VII, IXa$\}$, since \cref{conj:compat} is already known in the remaining cases by \cref{thm:soudry}.

   According to \cref{thm:RW2}, each of these classes of representations has the property that $L(\pi, s)$ has no subregular poles. For IIIa and IVa, there may be poles, but they are never subregular; for VII, IXa and supercuspidals, there are no poles at all. So for these representations, we have $L^{\Nov}(\pi \times \St, s) = L(\pi, s-\tfrac{1}{2})L(\pi, s + \tfrac{1}{2})$. On the other hand, since the Langlands parameters of these representations have no 1-dimensional summands, we have $L(\pi \times \St, s) = L(\pi, s-\tfrac{1}{2})L(\pi, s + \tfrac{1}{2})$ by \cref{lem:langlandsfactor}. So \cref{conj:compat} holds for all these representations.
  \end{proof}

\section{Proof of Theorems B, C and D}

 \begin{proof}[Proof of \cref{thm:subregpole}]
  Let $\pi$ and $s_0$ be as in the theorem. If $\chi_\pi |\cdot|^{2s_0 + 1} \ne 1$, then Corollary \ref{cor:AimpliesC} shows that $s_0$ is an exceptional pole of $L(\pi, s)$ if and only if it is a pole of $\frac{L(\pi, s) L(\pi, s+1)}{L^{\Nov}(\pi \times \St, s+ \tfrac{1}{2})}$. By \cref{thm:compat}, which we have just proved, the denominator agrees with the Langlands $L$-factor $L(\pi \times \St, s+ \tfrac{1}{2})$. This completes the proof of \cref{thm:subregpole} when $\chi_\pi |\cdot|^{2s_0 + 1} \ne 1$.

  If $\chi_\pi |\cdot|^{2s_0 + 1} = 1$, then \cref{cor:AimpliesB} (combined with \cref{thm:compat}) shows that $s_0$ is not a pole of $\frac{L(\pi, s) L(\pi, s+1)}{L(\pi \times \St, s+ \tfrac{1}{2})}$. So we must check that $s_0$ is a subregular pole if and only if $\phi_{\pi}$ has a direct summand of the form $|\cdot|^{-(s_0 + 1/2)} \otimes \operatorname{sp}(2)$. This follows by a case-by-case check from \cref{thm:RW2} combined with the tables of Langlands parameters in \cite{robertsschmidt07}.
 \end{proof}

 \begin{proof}[Proof of \cref{thm:Novexpole}]
  We first suppose $\sigma$ is an irreducible principal series $i(\mu_2, \nu_2)$. Twisting $\pi$ appropriately, we may assume $\mu_2 = 1$; and the irreducibility gives $\nu_2 \ne |\cdot|^{\pm 1}$. Moreover, $s_0$ is such that $\chi_{\pi} \nu_2 |\cdot|^{2s_0} = 1$, and we may assume $s_0 = 0$.

  By \cref{prop:PScase}, $0$ is an exceptional pole of the Novodvorsky $L$-factor if and only if it is a subregular pole of $L(\pi, s)$. Moreover, the irreducibility of $\sigma$ shows that $\nu_2 \ne |\cdot|$, so $\chi_{\pi} |\cdot|^{2s_0 + 1} = \nu_2^{-1} |\cdot| \ne 1$. So, by the first case of \cref{thm:subregpole}, $0$ is an exceptional pole of $L(\pi\times \sigma, s)$ if and only if $\phi_{\pi}$ has a 1-dimensional trivial summand; and this in turn implies that $\phi_{\pi} \otimes \phi_{\sigma}$ also has such a summand, since $\phi_{\pi}\otimes \phi_{\sigma} = \phi_{\pi} \oplus \phi_{\pi \otimes \nu}$.

  Conversely, if $\phi_{\pi} \otimes \phi_{\sigma}$ has a trivial summand, then it must come from either $\phi_{\pi}$ or $\phi_{\pi \otimes \nu}$. If the former holds, then reversing the argument shows that $L(\pi \times \sigma, s)$ has an exceptional pole at $0$. However, since $\nu = \chi_{\pi}^{-1}$, the two factors are dual to each other, so $\phi_{\pi \otimes \nu}$ has a trivial summand if and only if $\phi_{\pi}$ does.

  We now suppose $\sigma$ is a special representation. Again, we may assume $\sigma = \operatorname{St} \otimes |\cdot|^{1/2}$, so we are now in the case $\chi_{\pi} |\cdot|^{2s_0 + 1} = 0$. By \cref{cor:AimpliesB}, $s_0$ is an exceptional pole of $L(\pi \times \sigma, s)$ if and only if it is a subregular pole of $L(\pi, s)$; and the second case of \cref{thm:subregpole} shows that this occurs if and only if $s_0$ is a pole of the $L$-factor of a 2-dimensional summand of $\phi_{\pi}$ of the form $|\cdot|^{-(s_0 + 1/2)} \otimes \operatorname{sp}(2)$. Since $\phi_{\pi}$ cannot have any 3-dimensional summands, there is a bijection between 2-dimensional summands of $\phi_{\pi}$ and 1-dimensional summands of $\phi_{\pi} \otimes \phi_{\sigma}$, sending $\rho \otimes \operatorname{sp}(2)$ to $\rho |\cdot|^{1/2} \otimes \operatorname{sp}(1)$. So we conclude that $s_0$ is an exceptional pole of $L(\pi \times \sigma, s)$ if and only if $\phi_{\pi} \otimes \phi_{\sigma}$ has a summand $|\cdot|^{-s} \otimes \operatorname{sp}(1)$.
 \end{proof}

 \begin{proof}[Proof of \cref{thm:dist} for non-supercuspidal $\sigma$]
  Suppose first that $\sigma = i(\mu, \nu)$ is an irreducible principal series representation. Twisting $\pi$ and $\sigma$ appropriately, we may assume that $s_0 = 0$, so $\mu\nu = \chi_{\pi}^{-1}$.

  Then we have
  \[ \Hom_H(\pi \otimes (\triv \boxtimes \sigma), \CC) \cong \Hom_H(\pi \otimes (\sigma \boxtimes \triv), \CC) = \Hom_H(\pi, \sigma^\vee \boxtimes \triv) =\Hom_{H_+}(\pi, \rho) \]
  where $H_+$ denotes the subgroup $\left(\stbt{\star}{\star}{0}{\star}, \star\right)$ of $H$, and $\rho$ the character $\left(\stbt{a}{\star}{0}{d}, \star\right) \mapsto |a/d|^{1/2} \mu^{-1}(a) \nu^{-1}(d)$. Our claim is that this space is non-zero if and only if $L(\pi \times \sigma, s)$ has an exceptional pole at 0; by \cref{prop:PScase}, the latter is equivalent to $L(\pi \times \mu, s)$ having a subregular pole at 0.

  Similarly, if $\sigma$ is the Steinberg representation and $\chi_{\pi} = 1$, then the natural map
  \[ \Hom_H(\pi, \St \boxtimes \triv) \to \Hom_H(\pi, \Sigma \boxtimes \triv)\]
  is an isomorphism, by \cite[Theorem 4.3]{piatetskishapiro97}. Again, the right-hand side can be interpreted as a space of $H_+$-invariant functionals, where we take $\rho$ the character $\left(\stbt{a}{\star}{0}{d}, \star\right) \mapsto  |a/d|$; and we want to show that this space is non-zero if and only if $L(\pi \times \St, s)$ has an exceptional pole at $s = 0$, which is equivalent to $L(\pi, s)$ having a subregular pole at $-\tfrac{1}{2}$, by \cref{cor:AimpliesB}.

  Following \S 4 of \cite{roesnerweissauer18}, we refer to elements of $\Hom_{H_+}(\pi, \rho)$, where $\rho$ is a character of $H_+$, as ``$(H_+, \rho)$-functionals''. The claim we need to prove is:

  \begin{quotation}
   Let $\rho$ be the character $\left(\stbt{a}{\star}{0}{d}, \star\right) \mapsto |a/d|^{1/2} \mu^{-1}(a) \nu^{-1}(d)$ of $H_+$, where $\mu, \nu$ are characters of $F^\times$ such that $\mu\nu = \chi_{\pi}^{-1}$. Then the space of $(H_+, \rho)$-functionals on $\pi$ is 1-dimensional if $L(\pi \times \mu, s)$ has a subregular pole at $s = 0$, and zero otherwise.
  \end{quotation}
  This follows from the results of \cite[\S 5]{roesnerweissauer18}.
 \end{proof}

\section{Proof of \texorpdfstring{\cref{thm:branching}}{Theorem E}}

 \subsection{Uniqueness for \texorpdfstring{$\GSp(4) \times \GL(2)$}{GSp(4) x GL(2)}}

  Let $\pi$, $\sigma$ be irreducible generic representations of $\GSp(4, F)$ and $\GL(2, F)$ respectively. Then, for any $s_0 \in \CC$, the map $\tilde{Z}_{s_0}: \cW(\pi) \otimes \cS(F^2) \otimes \cW(\sigma) \to \CC$ defined by
  \[
   (W_0, \Phi_1, W_2) \mapsto \left.\frac{Z(W_0, \Phi_1, W_2, s)}{L^{\Nov}(\pi \times \sigma, s)}\right|_{s = s_0}
  \]
  satisfies $\tilde{Z}_{s_0}\left(hW_0, h_1 \Phi_1, h_2 W_2\right)= |\det h|^{-s_0} \tilde{Z}_{s_0}(W_0, \Phi_1, W_2)$. In particular, it factors through the maximal quotient of $ \cS(F^2)$ on which $F^\times$ acts via $\nu |\cdot|^{-2s_0}$, where $\nu = (\chi_{\pi} \chi_{\sigma})^{-1}$. We are interested in the case $s_0 = 0$, $\nu = 1$, in which case this quotient is isomorphic to $\Sigma = i(|\cdot|^{1/2}, |\cdot|^{-1/2})$, via $\Phi \mapsto F^{\Phi}$. Thus we have $\tilde{Z}_{s_0}(W_0, \Phi_1, W_2) = \mathfrak{z}(W_0, F^{\Phi_1}, W_2)$ for some non-zero element $\mathfrak{z} \in \Hom_H(\pi \otimes (\Sigma \boxtimes \sigma), \CC)$.

  There is a left-exact sequence
  \[
   0 \to \Hom_H\left(
   \pi \otimes (\triv \boxtimes \sigma), \CC\middle) \xrightarrow{\ \alpha\ }
   \Hom_H\middle(\pi \otimes (\Sigma \boxtimes \sigma), \CC\middle)   \xrightarrow{\ \beta\ }
   \Hom_H\middle(\pi \otimes (\operatorname{St}\boxtimes \sigma), \CC\right)
  \]
  in which the first and third terms both have dimension $\le 1$, by the multiplicity-one results for $\operatorname{GSpin}$ groups proved in \cite{emorytakeda21} and the isomorphisms $G(F) \cong \operatorname{GSpin}(5)$ and $H \cong \operatorname{GSpin}(4)$. \cref{conj:branching}(a) asserts that the middle group in the above sequence is always 1-dimensional, so the element $\mathfrak{z}$ is a basis.

  \begin{remark}
   Note that there do exist examples in which the first and last terms are both nonzero -- one can construct such examples with  $\pi$ and $\sigma$ principal-series.
  \end{remark}

  \begin{proposition}
   \label{prop:exactseq}
   The element $\mathfrak{z}$ is in the image of $\alpha$ if and only if $s = 0$ is an exceptional pole of $L^{\Nov}(\pi \times \sigma, s)$.
  \end{proposition}

  \begin{proof}
   This is essentially a restatement of the definitions, since the $F^{\Phi}$ with $\Phi(0,0) = 0$ span the generic subrepresentation $\St \subset \Sigma$.
  \end{proof}

  If $\sigma$ is non-supercuspidal, and $s = 0$ is not an exceptional pole of the Novodvorsky $L$-factor, then we know from \cref{thm:dist} that in fact $\Hom_H\left(\pi \otimes (\triv \boxtimes \sigma), \CC\right) = 0$; so \cref{conj:branching}(a) follows in this case (that is, we have proved Theorem E(a)(ii)). Conversely, if we assume \cref{conj:branching}(a), then it follows that $\Hom_H\left(\pi \otimes (\triv \boxtimes \sigma), \CC\right)$ is non-zero if and only if $\mathfrak{z}$ is in the image of $\alpha$, so  \cref{conj:branching}(a) implies \cref{conj:dist}.

 \subsection{Proof of \cref{thm:branching}(a)(i)} We now prove \cref{thm:branching} in the case where $\chi_{\pi} = \tau^2$ for some smooth character $\tau$. Replacing $\pi$ and $\sigma$ with the twists $\pi \times \tau$ and $\sigma \times \tau^{-1}$, which does not change either the Hom-space or the zeta-integral, we may in fact suppose that $\chi_{\pi} = 1$. In this case we can regard $\pi$ as a representation of $G  / Z_G = \operatorname{PGSp}(4, F) \cong \SO(5, F)$, and $\Sigma \boxtimes \sigma$ as a representation of the subgroup $H / Z_G \cong \SO(4, F)$.

  We now apply the results of \cite{moeglinwaldspurger12} on branching laws for representations of special orthogonal groups. In \emph{op.cit.} a branching multiplicity $m(\sigma, (\sigma')^\vee)$ is defined for irreducible representations $\sigma$ of $\SO(d, F)$ and $\sigma'$ of $\SO(d', F)$, where $d > d'$ are any integers of differing parity. (The results of \emph{op.cit.} also cover non-split special orthogonal groups as well, but we do not need this here.) If $d = d' + 1$, then $m(\sigma, (\sigma')^\vee)$ is just $\dim \Hom_{\SO(d', F)}(\sigma, (\sigma')^\vee) = \dim \Hom_{\SO(d', F)}(\sigma \otimes \sigma', \CC)$; in the other extreme case, if $d' = 0$, then $m(\sigma, (\sigma')^\vee)$ is the space of Whittaker functionals on $\sigma$. 

  The Proposition stated in Section 1.3 of \cite{moeglinwaldspurger12} analyses these multiplicities when $\sigma$ and $\sigma'$ are (possibly reducible) parabolic inductions, in which case $m(\sigma, (\sigma')^\vee)$ still makes sense. For these results, suppose that $\sigma$ is induced from a representation $\pi_1 |\cdot|^{b_1} \times \dots \times \pi_t |\cdot|^{b_t} \times \sigma_0$ of the Levi subgroup $\GL(d_1, F) \times \dots \times \GL(d_t, F) \times \SO(d_0, F)$ of $\SO(d, F)$, where $d = 2(d_1 + \dots + d_t) + d_0$, $\pi_i$ is a tempered irreducible representation of $\GL(d_i, F)$, $\sigma_0$ is a tempered irreducible representation of $\SO(d_0, F)$, and $b_1 \ge \dots \ge b_t \ge 0$ are real numbers. (The case $d_0 = 0$ or $1$ is allowed, in which case we understand $\SO(d_0)$ to be the trivial group.) We also make the same assumptions \emph{mutatis mutandis} for  $\sigma'$. The Proposition stated in \S 1.3 of \cite{moeglinwaldspurger12} (and proved in \S 1.3--1.8 of \emph{op.cit.}) shows that $m(\sigma, (\sigma')^\vee)$ is given by $m(\sigma_0, (\sigma_0')^\vee)$ if $d_0 > d_0'$, or $m(\sigma_0', (\sigma_0)^\vee)$ if $d_0 < d_0'$; in particular, since these numbers are known to be $\le 1$ (by the results quoted in the introduction of \emph{op.cit.}), we have $m(\sigma, (\sigma')^\vee) \le 1$.

  This class of parabolically-induced representations includes all generic irreducible representations; but it also contains some reducible representations -- crucially, the reducible representations of $\SO(4, F)$ we are calling $\Sigma \boxtimes \sigma$, for any generic irreducible representation of $\SO(3, F) \cong \operatorname{PGL}(2, F)$, or $\Sigma \boxtimes \Sigma$, both have this form. Hence, applying this result with $d = 5$, $d' = 4$, and the $\sigma$ and $\sigma'$ of \emph{op.cit.} taken to be our $\pi$ and $\Sigma \boxtimes \sigma$, we have $\dim \Hom_{\SO(4, F)}(\pi \otimes (\Sigma \boxtimes \sigma), \CC) \le 1$ as required.

 \subsection{Uniqueness for $\GSp(4)$}

  We also have a slight strengthening of the above result in the case when $\sigma$ is itself a twist of the Steinberg representation. Via twisting, we shall take $s_0 = 0$ and $\chi_{\pi}$ trivial, and consider the space $\Hom_H(\pi \otimes (\Sigma \boxtimes \Sigma), \CC)$. The argument of Moeglin--Waldspurger quoted above also applies in this situation, showing that shows that this space always has dimension 1.

  Let us write $\Xi = \Sigma \boxtimes \Sigma$, and filter it as $\Xi_{00} \subset \Xi_0 \subset \Xi$ where $\Xi_{00} = \St \boxtimes \St$, $\Xi_0 / \Xi_{00} = \left(\St \boxtimes \triv\right)\oplus \left(\triv \boxtimes \St\right)$ and $\Xi / \Xi_0 = \triv\boxtimes\triv$.

  \begin{proposition}
   The space $\Hom_H(\pi \otimes \Xi, \CC)$ contains a canonical non-zero homomorphism $\mathfrak{z}$ satisfying
   \[ \mathfrak{z}(W_0, F^{\Phi_1}, F^{\Phi_2}) = \left. \frac{Z(\widetilde{B}_{W_0}, \Phi_1, \Phi_2; \Lambda, s)}{L(\pi, s)}\right|_{s = -{1}/{2}},\qquad \Lambda = (1, 1). \]
   Its restriction to $\pi \otimes \Xi_{00}$ is non-trivial if and only if $s = -\tfrac{1}{2}$ is not a subregular pole of $L(\pi, s)$, in which case $\Hom_H(\pi \otimes \Xi, \CC)$ is 1-dimensional spanned by $\mathfrak{z}$, and every non-generic subquotient $\xi$ of $\Xi$ satisfies $\Hom_H(\pi \otimes \xi, \CC) = 0$.
  \end{proposition}

  \begin{proof}
   One checks easily that the zeta-integral $Z(\tilde{B}_{W_0}, \dots)$ depends only on the image of $\Phi_i$ in the $F^\times$-coinvariants, or equivalently on $F^{\Phi_i}$. Moreover, the fact that $\mathfrak{z}$ restricts non-trivially to $\Xi_0$ is precisely \cite[Theorem 4.3]{piatetskishapiro97}; and its proof moreover shows that $\Hom_H(\pi, \CC) = 0$ for generic $\pi$.

   If $s = -\tfrac{1}{2}$ is not a subregular pole, then \cref{thm:dist} shows that $\Hom_H(\pi \otimes (\triv \boxtimes \St),\CC)$ and $\Hom_H(\pi \otimes (\St \boxtimes \triv),\CC)$ are zero. Hence the restriction map $\Hom_H(\pi \otimes \Xi, \CC) \to \Hom_H(\pi \otimes \Xi_{00}, \CC)$ is injective. Since the latter space has dimension $\le 1$ by \cite{waldspurger12} the result follows.
  \end{proof}

\newcommand{\noopsort}[1]{}
\providecommand{\bysame}{\leavevmode\hbox to3em{\hrulefill}\thinspace}
\providecommand{\MR}[1]{%
 MR \href{http://www.ams.org/mathscinet-getitem?mr=#1}{#1}.
}
\providecommand{\href}[2]{#2}
\newcommand{\articlehref}[2]{\href{#1}{#2}}

\end{document}